\documentclass[12pt]{article}
\usepackage[centertags]{amsmath}
\usepackage{amsfonts}
\usepackage{amssymb}
\usepackage{latexsym}
\usepackage{amsthm}
\usepackage{newlfont}
\usepackage{graphicx}
\usepackage{listings}
\usepackage{booktabs}
\usepackage{abstract}
\usepackage{enumerate}
\usepackage{xcolor}
\RequirePackage{srcltx}
\date{}

\newlength{\defbaselineskip}
\newcommand{\setlinespacing}[1]%
           {\setlength{\baselineskip}{#1 \defbaselineskip}}

\newcommand{\actaqed}{\hfill $\actabox$}
{\medskip\noindent \textit{Proof of #1. }}%
{\actaqed \medskip}

\def\D{{\mathcal D}}

\def\cC{{\mathcal C}}

\def \cX{\mathcal X}

\def \T{\mathbb T}

\def \<{\langle}
\def\>{\rangle}

\def \e{\varepsilon}

\def \de{\delta}

\def \ff{\varphi}

\def \sp{\operatorname{span}}

\def\bx{\mathbf x}
\def\by{\mathbf y}

\def\bX{\mathbf X}
\def\bW{\mathbf W}
\def\bH{\mathbf H}

\def\btt{\mathbf t}

\def\bbE{{\mathbb E}}

\def\bbN{{\mathbb N}}

\def\bbP{{\mathbb P}}

\newtheorem{Theorem}{Theorem}[section]
\newtheorem{Lemma}{Lemma}[section]
\newtheorem{Proposition}{Proposition}[section]
\newtheorem{Corollary}{Corollary}[section]
\theoremstyle{definition}
\newtheorem{Definition}{Definition}[section]
\newtheorem{Remark}{Remark}[section]
\newtheorem{Example}{Example}[section]
\numberwithin{equation}{section}

\newcommand{\be}{\begin{equation}}
\newcommand{\ee}{\end{equation}}

\def\f{\frac}

\def\NN{{\mathbb N}}

\def\sub{\substack}

\def\cX{\mathcal{X}}

\def\spn{\operatorname{span}}
\def\bx{\mathbf{x}}

\def\by{\mathbf{y}}

\def\bx{\mathbf{x}}

\DeclareFontEncoding{FMS}{}{}
\DeclareFontSubstitution{FMS}{futm}{m}{n}
\DeclareFontEncoding{FMX}{}{}
\DeclareFontSubstitution{FMX}{futm}{m}{n}
\DeclareSymbolFont{fouriersymbols}{FMS}{futm}{m}{n}
\DeclareSymbolFont{fourierlargesymbols}{FMX}{futm}{m}{n}
\DeclareMathDelimiter{\VT}{\mathord}{fouriersymbols}{152}{fourierlargesymbols}{147}

\begin{document}

\title{Bounds for the sampling discretization error and their applications to the universal sampling discretization}

\author{E.D. Kosov and   V.N. Temlyakov 	
}

\newcommand{\Addresses}{{
  \bigskip
  \footnotesize


E.D. Kosov, \\
\textsc{Centre de Recerca Matem\`atica, Campus de Bellaterra, Edifici C 08193
Bellaterra (Barcelona), Spain.\\
E-mail:} \texttt{ked$_{-}$2006@mail.ru}

\medskip
V.N. Temlyakov, \\ \textsc{ Steklov Mathematical Institute of Russian Academy of Sciences, Moscow, Russia;\\ Lomonosov Moscow State University; \\ Moscow Center of Fundamental and Applied Mathematics; \\ University of South Carolina.
\\ E-mail:} \texttt{temlyakovv@gmail.com}

}}
\maketitle

\begin{abstract}
{
In the first part of the paper we study absolute error of sampling discretization
of the integral $L_p$-norm for function classes of continuous functions.
We use basic approaches from chaining technique to provide general upper bounds for the error of sampling discretization
of the $L_p$-norm on a given function class in terms of entropy numbers in the uniform norm of this class.
As an example we apply these general results to obtain new error bounds for sampling discretization of the $L_p$-norms on
classes of multivariate functions with mixed smoothness.
In the second part of the paper we apply our general bounds to study the problem of universal sampling discretization.
}
\end{abstract}

{\it Keywords and phrases}: Sampling discretization, Entropy, Universality.

{\it MSC classification:} Primary 65J05; Secondary 42A05, 65D30, 41A63.

\section{Introduction}
\label{I}

This paper belongs to the actively developing area of research -- sampling discretization and recovery. A number of interesting results was recently obtained in this area, e.g., see \cite{BSU}, \cite{CoDo}, \cite{CM}, \cite{DTM1}, \cite{DTM3}, \cite{DKU}, \cite{JUV}, \cite{KUV}, \cite{KU}, \cite{KU2}, \cite{KPUU}, \cite{KPUU2}, \cite{KNU24}, \cite{NSU}, \cite{VT183},
three survey papers \cite{DPTT}, \cite{KKLT}, \cite{LMT}, and citations therein.


To formulate our new results let $\Omega$ be a compact subset of $\mathbb{R}^d$, $\mu$ be a probability measure on $\Omega$,
and let $W\subset L_p(\Omega,\mu)\cap C(\Omega)$,
$1\le p<\infty$, be a class of continuous on $\Omega$ functions.
In this paper we are studying upper bounds for
the following optimal errors of discretization of the $L_p$ norm of functions from $W$
({\it discretization numbers} of $W$)
$$
er_m(W,L_p):= \inf_{\xi^1,\dots,\xi^m} \sup_{f\in W} \Bigl|\|f\|_p^p - \frac{1}{m}\sum_{j=1}^m |f(\xi^j)|^p\Bigr|.
$$
The special case when $W$ is a unit $L_p$ ball of some finite dimensional subspace was already thoroughly studied
(see \cite{DKT}, \cite{DT22}, \cite{Kos}, \cite{LT}).
The main aim of this paper is to establish general connections between
the sequence of optimal errors of discretization
and the sequence of the entropy numbers of a given functional class. This continues the classical direction of proving inequalities between asymptotic characteristics of function classes. One of the most famous results in this area is the Carl's inequality (see \cite{Ca}).
We recall that for a set $W$ in a Banach space $X$ (or in a linear space with the semi-norm $\|\cdot\|_X)$
the entropy numbers are defined as follows:
$$
\e_k(W,X) :=  \inf \bigl\{\e>0\colon \exists y^1,\dots ,y^{2^k} \in W\colon
W \subseteq \cup_{j=1}^{2^k} B_X(y^j,\e)\bigr\}
$$
where $B_X(y,\e):=\{x\in X\colon \|x-y\|_X \le \e\}$.
In our definition of $\e_k(W,X)$ we require $y^j\in W$.
In the standard one this restriction is not imposed. However, it is well known that these characteristics may differ at most
by the factor 2 (see \cite[page 208]{VTbook}).

The connection between the discretization numbers and the entropy numbers
has already been studied (e.g. see \cite{VT191} and  \cite{VT171})
and the following result was proved in \cite{VT191}.

\begin{Theorem}\label{deT1}
Let $r\in (0, 1/2)$, $b\ge 0$.
Assume that a class of real functions $W$ is such that for all $f\in W$ we have $\|f\|_\infty \le M$
with some constant $M$. Also assume that the entropy numbers of $W$ in the uniform norm $L_\infty$ satisfy the condition
$$
  \e_n(W, L_\infty) \le Bn^{-r}[\log(n+1)]^b,
  \quad \forall n\in \mathbb{N}.
$$
Then
$$
er_m(W, L_2)  \le C(B, M, r, b)m^{-r}[\log(m+1)]^b, \quad \forall m\in \mathbb{N}.
$$
\end{Theorem}

We note that, in the case $r=1/2$, the same proof gives
\be\label{de1}
er_m(W,L_2)  \le C(B, M, b)m^{-1/2} [\log(m+1)]^{b+1},
\quad \forall m\in \mathbb{N}.
\ee

In Section \ref{gub} of this paper, we prove the following generalization of Theorem \ref{deT1}.

\begin{Theorem}\label{IT2}
Let $p\in[1, \infty)$, $r\in (0, 1/2]$, $b\ge 0$.
Assume that a class of real functions $W\subset C(\Omega)$ is such that for all $f\in W$ we have $\|f\|_\infty \le M$
with some constant $M$.
Also assume that the entropy numbers of $W$ in the uniform norm $L_\infty$ satisfy the condition
$$
\varepsilon_n(W, L_\infty)
\le Bn^{-r}[\log(n+1)]^b, \quad \forall n\in \mathbb{N}.
$$
Then we have
\begin{align}
1. &\hbox{ for } r\in(0, 1/2)\colon
er_m(W, L_p)\le Cm^{-r}[\log (m+1)]^{b}, \quad
\forall m\in \mathbb{N};\nonumber
\\
2. &\hbox{ for } r=1/2\colon
er_m(W, L_p)\le Cm^{-1/2}[\log (m+1)]^{b+1}, \quad
\forall m\in \mathbb{N},\nonumber
\end{align}
where $C:=C(B, M, r, b, p)>0$.
\end{Theorem}

Moreover,
we actually get the following general abstract result connecting the discretization numbers
with the entropy numbers of a function class.

\begin{Theorem}\label{KT0}
Let $p\in[1,+\infty)$.
Let $W\subset C(\Omega)$ be such that
$\|f\|_\infty\le M$ $\forall f\in W$.
Then
$$
er_m(W, L_p)\le Cp m^{-1/2}M^{p-1}\Bigl(M+ \sum\limits_{n=0}^m(n+1)^{-1/2} \e_n(W, L_\infty)\Bigr)
$$
where $C>0$ is a numerical constant.
\end{Theorem}

The proof of this theorem involves a simple application of the Dudley's entropy bound from the generic chaining technique.

We also apply Theorems \ref{IT2} and \ref{KT0}
to the classes $\bW^r_q$ and $\bH^r_q$ of multivariate periodic functions with small smoothness.
In the second half of Section~\ref{gub} we study other general results connecting bounds for the error of sampling discretization
with bounds for the entropy numbers in the uniform norm (see Theorems \ref{LpT1} and \ref{KT1} below).




In Section \ref{udr} we apply the results of Section \ref{gub} to the following problem of universal
sampling discretization.

\begin{Definition}\label{ID1} Let $0< p<\infty$. We say that a set $\xi:= \{\xi^j\}_{j=1}^m \subset \Omega$
provides {\it  $L_p$-universal sampling discretization} for the collection $\cX:= \{X(n)\}_{n=1}^k$ of finite-dimensional
linear subspaces $X(n)$ if we have
\be\label{ud}
\frac{1}{2}\|f\|_p^p \le \frac{1}{m} \sum_{j=1}^m |f(\xi^j)|^p\le \frac{3}{2}\|f\|_p^p\quad
\text{for any}\quad f\in \bigcup_{n=1}^k X(n).
\ee
We denote by $m(\cX,p)$ the minimal $m$ such that there exists a set $\xi$ of $m$ points, which
provides  the $L_p$-universal sampling discretization (\ref{ud}) for the collection $\cX$.

We will use a brief form $L_p$-usd for the $L_p$-universal sampling discretization (\ref{ud}).
\end{Definition}

Let $X$ be a Banach space with a norm $\|\cdot\|:=\|\cdot\|_X$, and let $\D=\{g_i\}_{i=1}^\infty $ be a given (countable)
system of elements in $X$. Given a finite subset $J\subset \NN$, we define $V_J(\D):=\spn\{g_j\colon  j\in J\}$.
For a positive  integer $v$, we denote by $\mathcal{X}_v(\D)$ the collection of all linear spaces $V_J(\D)$  with
$|J|=v$, and denote by $\Sigma_v(\D)$ the set of all $v$-term approximants with respect to $\D$; that is,
\begin{equation}\label{sigma-v}
\Sigma_v(\D):= \bigcup_{V\in\cX_v(\D)} V.
\end{equation}
We are interested in the $L_p$-universal sampling discretization for collections $\cX_v(\D)$, $\D = \D_N= \{g_1, \ldots, g_N\}$.
This problem was recently studied in \cite{DT} and \cite{DTM2} and
the following theorem was proved in \cite{DTM2}.

\begin{Theorem}\label{DT-Th}
Let $p\in [1,2]$ and $\D_N= \{g_1, \ldots, g_N\}$. Assume that
\begin{equation}\label{Cond1}
\|g_j\|_\infty\le 1\quad \forall j\in \{1,\ldots, N\},
\end{equation}
and assume that there is a constant $K\ge 1$ such that,
for any $(a_1, \ldots, a_N)\in \mathbb{C^N}$,
one has
\begin{equation}\label{Cond2}
\sum_{j=1}^{N}|a_j|^2\le K\Bigl\|\sum_{j=1}^{N}a_jg_j\Bigr\|^2_2.
\end{equation}
Then
$$
m(\cX_v(\D_N),p)\le C(p)Kv\log N[\log2Kv]^2
[\log 2Kv + \log\log N].
$$
\end{Theorem}

In Section \ref{udr} we will study the
$L_p$-universal sampling discretization for collections $\cX_v(\D_N)$ for $p\in (0, 2]$.
Our main result can be formulated as follows

\begin{Theorem}
Let $p\in(0, 2]$, and assume that the conditions \eqref{Cond1} and \eqref{Cond2} above are valid
for the system $\D_N$.
Then
$$
m(\cX_v(\D_N), p)\le C(p)Kv\log N
[\log 2Kv + \log\log N]^3.
$$
Moreover, for any $\varepsilon\in (0, 1)$, there are
$$
m\le C(p)\varepsilon^{-2}[\log\varepsilon^{-1}]^{3}Kv\log N[\log 2Kv + \log\log 2N]^3,
$$
points $\{\xi^1, \ldots, \xi^m\}\subset \Omega$ such that
$$
(1-\varepsilon)\|f\|_p^p\le \frac{1}{m}\sum_{j=1}^{m}|f(\xi^j)|^p\le (1+\varepsilon)\|f\|_p^p
\quad \forall f\in \Sigma_v(\D_N).
$$
\end{Theorem}

We note that in the range $p\in [1, 2]$, this result provides a slightly worse bound for $m(\mathcal{X}_v(\D_N),p)$
compared to Theorem \ref{DT-Th}. However, it provides a new bound for $m(\mathcal{X}_v(\D_N),p)$ in the range $p\in(0, 1)$.
In addition, the dependence of the sufficient for a good discretization number of points on the parameter $\varepsilon$
is significantly better compared to the results from \cite{DT} and \cite{DTM2} (see Theorem 5.1 and Corollary 5.2 in \cite{DT} and Theorem 6.1 in \cite{DTM2}).
Moreover, the approach to all the values of $p\in(0, 2)$ is the same and is based on one of the general upper
bounds for the error of sampling discretization, obtained in this paper.
This gives a new way of analyzing the problem of universal sampling discretization.






\section{Bounds for the error of sampling discretization and applications}
\label{gub}

This section is devoted to providing two general upper bounds for the error of sampling discretization.

\subsection{Preliminaries}

We start with some notation and preliminary known results, that will be used further in the proofs.

Throughout this section it will be more convenient to use the following sparsified collection of entropy numbers instead of
the whole set of entropy numbers defined in the introduction.
\begin{Definition}\label{ent-2}
Let $W$ be a compact subset of some metric space with a metric $d$, and
let $k\in\mathbb{N}\cup\{0\}$.
Then
$$
e_k(W, d):=
\inf\Bigl\{\varepsilon\colon \exists y^1,\ldots, y^{N_k}\in W\colon
W\subset \bigcup\limits_{j=1}^{N_k}B(y^j, \varepsilon)\Bigr\},
$$
where $N_k=2^{2^k}$ for $k\ge 1$ and $N_0=1$.
\end{Definition}
If the metric $d$ is induced by a norm $\|\cdot\|_E$ of some linear space $E$, we will
also use the notation $e_k(W,\|\cdot\|_E)$ in place of $e_k(W, d)$.
We note that, in this case, $e_0(W, \|\cdot\|_E)=\varepsilon_0(W, E)$ and $e_k(W, \|\cdot\|_E) = \varepsilon_{2^k}(W, E)$.


\begin{Definition}
An admissible sequence of $W$ is an increasing sequence $(\mathcal{W}_k)$ of partitions of $W$
such that $|\mathcal{W}_k|\le 2^{2^k}$ for all $k\ge1$ and $|\mathcal{W}_0|=1$.
For $f\in W$ let $W_k(f)$ denote the unique element of $\mathcal{W}_k$ that contains $f$.
\end{Definition}

\begin{Definition}
Let
$$
\gamma_2(W,d):=
\inf\sup_{f\in W}\sum\limits_{k=0}^\infty 2^{k/2} \mathrm{diam}\bigl(W_k(f)\bigr),
$$
where $\mathrm{diam}(F):=\sup\limits_{f,g\in F}d(f,g)$
and where the infimum is taken over all admissible sequences of $W$ .
\end{Definition}
The quantity $\gamma_2(W,d)$ is called the chaining functional.
If the metric $d$ is induced by a norm $\|\cdot\|_E$ of some linear space $E$, we will
also use the notation $\gamma_2(W, \|\cdot\|_E)$ in place of $\gamma_2(W, d)$.

The following lemma is the classical generic chaining bound.

\begin{Lemma}[see Theorem 2.2.18 in \cite{Tal}]\label{chaining}
Let $(W, d)$ be a metric space.
Assume that for a random process $\{V_f\}_{f\in F}$ and
for every pair $f, g\in W$ one has
$$
\bbP\Bigl(|V_f - V_g|\ge t d(f, g)\Bigr)\le
2e^{-t^2/2}\quad \forall t>0.
$$
Then there is a constant $C>0$ such that
for any $f_0\in W$, one has
$$
\mathbb{E}\bigl[\sup\limits_{f\in W}|V_f - V_{f_0}|\bigr]
\le C\gamma_2(W, d).
$$
\end{Lemma}

The classical Dudley's entropy bound states that (see \cite[Proposition 2.2.10]{Tal})
\begin{equation}\label{Dudley}
\gamma_2(W,d)\le \sum\limits_{k=0}^\infty 2^{k/2} e_k(W,d).
\end{equation}




To apply the generic chaining bound we will use the following tails bound for the Bernoulli process.

\begin{Lemma}[see Lemma 4.3 in \cite{LedTal}]\label{TailsEst}
Let $\eta_1,\ldots, \eta_m$ be independent symmetric Bernoulli random variables with values $\pm1$.
Then there is a number $C>0$ such that
$$
\mathbb{P}\Bigl(\bigl|\sum\limits_{j=1}^m\eta_j\alpha_j\bigr|\ge
C\bigl(\sum\limits_{j=1}^m|\alpha_j|^2\bigr)^{1/2}t\Bigr)\le 2e^{-t^2/2}.
$$
\end{Lemma}



The following lemma is just a combination of some known entropy numbers properties of subsets of finite dimensional spaces
and some straightforward computations (e.g. see Lemma 2.15 in \cite{KT24}).

\begin{Lemma}\label{lem-bound}
Let $a, b>0$.
Let $W$ be a subset of some $m$-dimensional space endowed with a norm $\|\cdot\|$.
Then there is a number $C(a, b)>0$ such that
$$
\sum\limits_{k > [\log m]}\bigl(2^{a k} e_k(W, \|\cdot\|)\bigr)^b
\le C(a, b) \sum\limits_{k \le [\log m]}\bigl(2^{a k} e_k(W, \|\cdot\|)\bigr)^b.
$$
\end{Lemma}


\subsection{The first general upper bound for the error of sampling discretization}

We will deduce Theorem \ref{KT0} from the following result.

\begin{Theorem}\label{KT2}
Let $W\subset \cC(\Omega)$ and let $\Phi\colon \mathbb{R}\to [0, +\infty)$
be such that
$$
|\Phi(f(x)) - \Phi(g(x))|\le L|f(x) - g(x)|\quad \forall f, g\in W, \forall x\in \Omega.
$$
There is a numerical constant $c>0$ such that for any
i.i.d. random vectors $X_1, \ldots, X_m$ with the distribution $\mu$, one has
\begin{multline*}
\mathbb{E}\Bigl[\sup\limits_{f\in W}
\Bigl|\frac{1}{m}\sum\limits_{j=1}^m\Phi(f(X_j)) - \int_\Omega\Phi(f)\, d\mu\Bigr|\Bigr]
\\
\le
cm^{-1/2}L\mathbb{E}\Bigl(\sup\limits_{f\in W}\max\limits_{1\le j\le m}|f(X_j)|+
\sum\limits_{k\le [\log m]} 2^{k/2} e_k(W, \|\cdot\|_{\infty, \mathbf{X}})\Bigr),
\end{multline*}
where
$\|g\|_{\infty, {\mathbf{X}}} := \max\limits_{1\le j\le m}|g(X_j)|$.
\end{Theorem}

\begin{proof}
By \cite[Lemma 9.1.11]{Tal} we have
$$
\mathbb{E}\Bigl[\sup\limits_{f\in W}
\Bigl|\frac{1}{m}\sum\limits_{j=1}^m\Phi(f(X_j)) - \int_\Omega\Phi(f)\, d\mu\Bigr|\Bigr]
\le
2m^{-1}\mathbb{E}\Bigl[\sup\limits_{f\in W} \Bigl|\sum\limits_{j=1}^m\eta_j\Phi(f(X_j))\Bigr|\Bigr]
$$
where $\eta_1,\ldots, \eta_m$ are independent symmetric Bernoulli random variables with values $\pm1$.
To estimate
$$
\mathbb{E}_{\eta_1, \ldots, \eta_m}\Bigl[\sup\limits_{f\in W} \Bigl|\sum\limits_{j=1}^m\eta_j\Phi(f(X_j))\Bigr|\Bigr]
$$
for fixed points $X_1, \ldots, X_m$
we want to apply the generic chaining bound from Lemma \ref{chaining} to the Bernoulli
random process $V_f = \sum_{j=1}^{m}\eta_j\Phi(f(X_j))$.
By Lemma \ref{TailsEst} and Lemma \ref{chaining}
we have
$$
\mathbb{E}\bigl[\sup\limits_{f\in W}|V_f - V_{f_0}|\bigr]\le C\gamma_2(W, d),
$$
where
$$
d(f, g)= \Bigl(\sum\limits_{j=1}^m|\Phi(f(X_j)) - \Phi(g(X_j))|^2\Bigr)^{1/2}\le
\sqrt{m}L\|f - g\|_{\infty, \mathbf{X}}
$$
and where
$$
\|h\|_{\infty, \mathbf{X}}:=\max\limits_{1\le j\le m}|h(X_j)|.
$$
By the Dudley's entropy bound \eqref{Dudley}
we have
$$
\gamma_2(W, d)\le \sum\limits_{k=0}^\infty 2^{k/2} e_k(W, d)
\le \sqrt{m}L\sum\limits_{k=0}^\infty 2^{k/2} e_k(W, \|\cdot\|_{\infty, \mathbf{X}})
$$
Thus,
\begin{multline*}
\mathbb{E}_{\eta_1, \ldots, \eta_m}\Bigl[\sup\limits_{f\in W} \Bigl|\sum\limits_{j=1}^m\eta_j\Phi(f(X_j))\Bigr|\Bigr]
= \mathbb{E}\bigl[\sup\limits_{f\in W}|V_f|\bigr]
\le \mathbb{E}\bigl[\sup\limits_{f\in W\cup\{0\}}|V_f|\bigr]
\\
=\mathbb{E}\bigl[\sup\limits_{f\in W\cup\{0\}}|V_f - V_0|\bigr]
\le C\sqrt{m}L\sum\limits_{k=0}^\infty 2^{k/2} e_k(W\cup\{0\}, \|\cdot\|_{\infty, \mathbf{X}})
\\
\le C_1\sqrt{m}L\Bigl(\sup\limits_{f\in W}\max\limits_{1\le j\le m}|f(X_j)|+
\sum\limits_{k=0}^\infty 2^{k/2} e_k(W, \|\cdot\|_{\infty, \mathbf{X}})\Bigr)
\end{multline*}
since
$$
e_0(W\cup\{0\}, \|\cdot\|_{\infty, \mathbf{X}})\le \sup\limits_{f\in W}\max\limits_{1\le j\le m}|f(X_j)|
$$
and
$$
e_k(W\cup\{0\}, \|\cdot\|_{\infty, \mathbf{X}})\le e_{k-1}(W, \|\cdot\|_{\infty, \mathbf{X}})\quad \forall k\ge 1.
$$
We now note that, by Lemma \ref{lem-bound},
since the set
$$
W_{\mathbf{X}}:=\{(f(X_1), \ldots, f(X_m))\colon f\in W\}
$$
is a subset of an $m$-dimensional space, we have
$$
\sum\limits_{k=0}^\infty 2^{k/2} e_k(W, \|\cdot\|_{\infty, \mathbf{X}})
\le
C_2\sum\limits_{k\le [\log m]} 2^{k/2} e_k(W, \|\cdot\|_{\infty, \mathbf{X}}).
$$
Thus,
\begin{multline*}
\mathbb{E}\Bigl[\sup\limits_{f\in W}
\Bigl|\frac{1}{m}\sum\limits_{j=1}^m\Phi(f(X_j)) - \int_\Omega\Phi(f)\, d\mu\Bigr|\Bigr]
\\
\le
C_3m^{-1/2}L\mathbb{E}\Bigl(\sup\limits_{f\in W}\max\limits_{1\le j\le m}|f(X_j)|+
\sum\limits_{k\le [\log m]} 2^{k/2} e_k(W, \|\cdot\|_{\infty, \mathbf{X}})\Bigr)
\end{multline*}
as announced.
\end{proof}

We now deduce Theorem \ref{KT0} from the result above.

{\bf Proof of Theorem \ref{KT0}.}

We note that for $p\ge 1$ and for a function $\Phi(f) = |f|^p$,
one has
\begin{multline*}
|\Phi(f(x)) - \Phi(g(x))|= \bigl| |f(x)|^p - |g(x)|^p \bigr|
\\
\le p\bigl[\max\{|f(x)|, |g(x)|\}\bigr]^{p-1}|f(x) - g(x)|
\le pM^{p-1}|f(x) - g(x)|.
\end{multline*}
By Theorem \ref{KT2},
$$
er_m(W, L_p)\le
cm^{-1/2}pM^{p-1}\mathbb{E}\Bigl(\sup\limits_{f\in W}\max\limits_{1\le j\le m}|f(X_j)|+\!\!
\sum\limits_{k\le [\log m]} 2^{k/2} e_k(W, \|\cdot\|_{\infty, \mathbf{X}})\Bigr).
$$
Since
$$
\sup\limits_{f\in W}\max\limits_{1\le j\le m}|f(X_j)|\le M
$$
and
$$
e_k(W, \|\cdot\|_{\infty, \mathbf{X}}) \le e_k(W, \|\cdot\|_\infty)
$$
we get the estimate
$$
er_m(W, L_p)\le
cm^{-1/2}pM^{p-1}\Bigl(M+ \sum\limits_{k\le [\log m]} 2^{k/2} e_k(W, \|\cdot\|_\infty)\Bigr).
$$
Finally, we note that
\begin{align*}
&\sum\limits_{k\le [\log m]}2^{k/2}e_k(W, \|\cdot\|_\infty)
\\
&\le \e_0(W, L_\infty) +
2\sum\limits_{k=1}^{[\log m]} \sum_{n= 2^{k-1}+1}^{2^{k}}n^{-1/2}\e_n(W, L_\infty)
\\
&\le
\e_0(W, L_\infty) +
2\sum_{n = 2}^{m}n^{-1/2}\e_n(W, L_\infty)
\\&
\le
2\sqrt{2}\sum_{n = 0}^{m}(n+1)^{-1/2}\e_n(W, L_\infty),
\end{align*}
which implies the announced bound. \qed

We now prove Theorem \ref{IT2}.

{\bf Proof of Theorem \ref{IT2}.}
By Theorem \ref{KT0}, we have
$$
er_m(W, L_p)\le
cm^{-1/2}pM^{p-1}\Bigl(M+ \sum\limits_{k\le [\log m]} 2^{k/2} e_k(W, \|\cdot\|_\infty)\Bigr).
$$
Under the assumptions of Theorem \ref{IT2},
$$
\sum\limits_{k\le [\log m]}2^{k/2}e_k(W, \|\cdot\|_\infty)
\le 2M + B
\sum\limits_{k= 1}^{[\log m]}2^{k\frac{1 - 2r}{2}}(k+1)^b.
$$

$1.$ If $r\in(0, 1/2)$ then $1-2r>0$ and
\begin{multline*}
2 M
+ B
\sum\limits_{k = 1}^{[\log m]}2^{k\frac{1 - 2r}{2}}(k+1)^b
\le
2M
+ B(1+\log (m+1))^b\frac{2m^{\frac{1 - 2r}{2}}}{2^{\frac{1 - 2r}{2}}-1}
\\
\le
C_1(M, B, r, b)[\log (m+1)]^bm^{\frac{1 - 2r}{2}}.
\end{multline*}
By Theorem \ref{KT0} we have
\begin{multline*}
er_m(F, L^p)\le cpM^{p-1}m^{-1/2}
C_2(M, B, r, b)[(\log (m+1)]^bm^{\frac{1 - 2r}{2}}
\\= C_2(M, B, p, r, b)[\log (m+1)]^bm^{-r}.
\end{multline*}

$2.$ If $r=1/2$ then $1-2r=0$ and
\begin{multline*}
2M + B \sum\limits_{k= 1}^{[\log m]}2^{k\frac{1 - 2r}{2}}(k+1)^b
\\
\le
2 M
+ B(1+\log (m+1))^{1+b}
\le
C_3(M, B, r, b)[\log (m+1)]^{b+1}.
\end{multline*}
By Theorem \ref{KT0} we have
\begin{multline*}
er_m(F, L^p)\le CpM^{p-1}m^{-1/2}
C_4(M, B, r, b)[\log (m+1)]^{b+1}
\\=
C_5(M, B, p, r, b)m^{-1/2}[\log (m+1)]^{b+1}.
\end{multline*}
The theorem is proved. \qed

\subsection{Examples: discretization for classes with small mixed smoothness}
\label{C}

Throughout this subsection, $\Omega =\T^d:=[0,2\pi)^d$ and $\mu$ is the normalized Lebesgue measure on $\Omega$.
Here, as an example, we apply Theorem \ref{IT2} to classes $\bW^r_q$ and $\bH^r_q$ of multivariate periodic functions with small smoothness, since there are known bounds for the entropy numbers with respect to the uniform norm
for these classes for certain values of parameters.

We start introducing these classes.

\begin{Definition}
In the univariate case, for $r>0$, let
\be\label{sr7}
F_r(x):= 1+2\sum_{k=1}^\infty k^{-r}\cos (kx-r\pi/2)
\ee
and in the multivariate case, for $\bx=(x_1,\dots,x_d)$, let
$$
F_r(\bx) := \prod_{j=1}^d F_r(x_j).
$$
Denote
$$
\bW^r_q := \{f:f=\varphi\ast F_r,\quad \|\varphi\|_q \le 1\},
$$
where
$$
(\varphi \ast F_r)(\bx):= (2\pi)^{-d}\int_{\T^d} \ff(\by)F_r(\bx-\by)d\by.
$$
\end{Definition}
The classes $\bW^r_q$ are classical classes of functions with {\it dominated mixed derivative}
(Sobolev-type classes of functions with mixed smoothness).
The reader can find results on approximation properties of these classes in the books \cite{VTbookMA} and \cite{DTU}.

We now define the classes $\bH^r_q$.
\begin{Definition}
Let  $\btt =(t_1,\dots,t_d )$ and $\Delta_{\btt}^l f(\bx)$
be the mixed $l$-th difference with
step $t_j$ in the variable $x_j$, that is
$$
\Delta_{\btt}^l f(\bx) :=\Delta_{t_d,d}^l\dots\Delta_{t_1,1}^l
f(x_1,\dots ,x_d ) .
$$
Let $e$ be a subset of natural numbers in $[1,d ]$. We denote
$$
\Delta_{\btt}^l (e) =\prod_{j\in e}\Delta_{t_j,j}^l,\qquad
\Delta_{\btt}^l (\varnothing) = I .
$$
We define the class $\bH_{q,l}^r B$, $l > r$, as the set of
$f\in L_p$ such that for any $e$
\be\label{C1}
\bigl\|\Delta_{\btt}^l(e)f(\bx)\bigr\|_p\le B
\prod_{j\in e} |t_j |^r .
\ee
\end{Definition}
In the case $B=1$ we omit it. It is known (see, for instance, \cite{VTbookMA}, p.137) that the classes
$\bH^r_{q,l}$ with different $l$ are equivalent. So, for convenience we fix one $l= [r]+1$ and omit $l$ from the notation.

It is well known and easy to check that for $r>1/q$ there exists a constant $C(d,r,q)$ such that for all
$f\in \bW^r_q$ and $f\in \bH^r_q$ we have $\|f\|_\infty \le C(d,r,q)$.
In \cite{VT191} (see Theorems 3.3. and 3.4 there), for $d\ge 2$,  $2<q\le\infty$, and $1/q<r<1/2$,
the following two estimates were obtained:
$$
\e_k(\bW^r_q,L_\infty) \ll k^{-r} (\log (k+1))^{(d-1)(1-r)+r}.
$$
and
$$
\e_k(\bH^r_q,L_\infty) \ll k^{-r} (\log (k+1))^{d-1+r}.
$$
In the case $r=1/2$, for $d\ge 2$ and $2<q\le\infty$, the following two bounds are known (see \cite{VT184} and \cite{VT185}):
$$
\e_k(\bW^{1/2}_q,L_\infty) \ll k^{-1/2} (\log (k+1))^{d/2}(\log\log k)^{3/2}
$$
and
$$
\e_k(\bH^{1/2}_q,L_\infty) \ll k^{-1/2} (\log (k+1))^{d-1/2}(\log\log k)^{3/2}.
$$
Thus, applying Theorem \ref{IT2} (or Theorem \ref{KT0} itself) we arrive at the following two statements.

\begin{Example}\label{CT5} Let $d\ge 2$, $p\in [1,\infty)$,  $2<q\le\infty$, and $1/q<r<1/2$. Then
$$
er_m(\bW^r_q,L_p) \ll m^{-r} (\log m)^{(d-1)(1-r)+r}.
$$
and
$$
er_m(\bH^r_q,L_p) \ll m^{-r} (\log m)^{d-1+r}.
$$
\end{Example}

We note that, in the case $p=2$, the result from Example \ref{CT5} was obtained in \cite{VT191}.

\begin{Example}\label{CT8} Let $d\ge 2$, $p\in [1,\infty)$, and $2<q\le\infty$. Then
$$
er_m(\bW^{1/2}_q,L_p) \ll m^{-1/2} (\log (m+1))^{d/2+1}(\log\log m)^{3/2}
$$
and
$$
er_m(\bH^{1/2}_q,L_p) \ll m^{-1/2} (\log (m+1))^{d+1/2}(\log\log m)^{3/2}.
$$
\end{Example}

\subsection{Bounds for the distribution function of the error of sampling discretization}

In Theorem \ref{KT2} we have obtained a bound for the expectation
$$
\mathbb{E}\Bigl[\sup\limits_{f\in W}|L_\bX(f)|\Bigr]
$$
of the error
$$
L_\bX(f):=L^\Phi_\bX(f) := \frac{1}{m}\sum_{j=1}^m \Phi(f(X_j)) -\int_\Omega \Phi(f)d\mu,\qquad \bX:= (X_1,\dots, X_m)
$$
when the function $\Phi$ satisfies the Lipschitz condition
\be\label{Lip}
|\Phi(u)-\Phi(v)| \le L|u-v|
\ee
and when $\bX = (X_1, \ldots, X_m)$, where random vectors $X_j$ are i.i.d. with the distribution $\mu$.
In this subsection we are going to prove an estimate for the
distribution function of $\sup\limits_{f\in W}|L_\bX(f)|$.
Our main tool will be the following classical Bernstein's inequality.
(see, for instance, \cite{VTbook}, p.198).
For $m$ independent copies $\xi_1, \ldots, \xi_m$ of a random variable $\xi$ we have:
If $|\xi-\bbE(\xi)|\le M$ a.s. then, for any $\e>0$,
\begin{equation}\label{31.6}
\mathbb{P}\Bigl\{\Bigl|\frac{1}{m}\sum_{j=1}^m\xi_j-\bbE(\xi)\Bigr|\ge\e\Bigr\} \le 2\exp\Bigl(-\frac{m\e^2}{2(\sigma^2(\xi)+M\e/3)}\Bigr),
\end{equation}
where $\sigma^2(\xi)$ is the variance of $\xi$.

\begin{Lemma}\label{Lemma 33.1}
If $\delta> 0$  and $\|f_1-f_2\|_\infty\le \delta$,
then for $\Phi$ satisfying \eqref{Lip}, for $\eta \le 4L\de$, we have
$$
\mathbb{P}\{|L_\bX(f_1)-L_\bX(f_2)|\ge\eta\}\le
2\exp\Bigl(-\frac{m\eta^2}{8L^2\delta^2}\Bigr).
$$
\end{Lemma}
\begin{proof} Consider the random variables $\xi_j=\Phi(f_1(X_j))-\Phi(f_2(X_j))$.
By \eqref{Lip}
$$
|\xi_j|\le L\de.
$$
Therefore, $|\xi_j-\bbE\xi_j|\le 2L\delta$ and the variance $\sigma^2(\xi_j)$ of $\xi_j$ is at
most $L^2\delta^2$. Applying the Bernstein inequality \eqref{31.6} to $\xi_j$ we get
\begin{eqnarray}\label{0.1}
\mathbb{P}\{|L_\bX(f_1)-L_\bX(f_2)|\ge\eta\}\nonumber
\le2\exp\Bigl(-\frac{m\eta^2}{2(L^2\delta^2+2L\delta\eta/3)}\Bigr) \nonumber \\
\le2\exp\Bigl(-\frac{m\eta^2}{2(11L^2\delta^2/3)}\Bigr), \nonumber
\end{eqnarray}
and Lemma \ref{Lemma 33.1} follows.
\end{proof}


\begin{Theorem}\label{LpT1} Let a function class $W$ be such that
\be\label{M}
\|f\|_\infty \le L, \quad \forall \, f\in W,
\ee
and
$$
\sum_{n=1}^\infty n^{-1/2}\e_n(W, L_\infty) =\infty.
$$
Suppose that $\Phi$ satisfies \eqref{Lip}.
For $\eta>0$ define $J:=J(\eta/L)$ as the minimal $j$ satisfying $\e_{2^{j}}(W, L_\infty)\le \eta/(4L)$ and
$$
S_J:= \sum_{j=1}^J2^{(j+1)/2}\e_{2^{j-1}}(W, L_\infty).
$$
Then for $m\in\bbN$, $\eta\in (0,L]$ satisfying $m(\eta/S_J)^2 \ge 2^7L^2$ we have
$$
\mathbb{P}\bigl\{\sup_{f\in W} |L_\bX(f)|\ge \eta\bigr\} \le 9\exp(-c(L)m(\eta/{\bar S}_J)^2)
$$
with $c(L):= (2^7L^2)^{-1}$ and ${\bar S}_J := \max(1,S_J)$.
\end{Theorem}
\begin{proof}
Denote $\de_j := \e_{2^j}(W, L_\infty)$, $j=0,1,\dots$, and consider minimal $\de_j$-nets ${\mathcal N}_j \subset W$ of $W$.
We will use the notation $N_j:= |{\mathcal N}_j|$.
Since $J$ is the minimal $j$ satisfying $\de_{j} \le \eta/(4L)$, then $\de_{j-1} > \eta/(4L)$ for $j=1,\dots,J$.
For $j=1,\dots,J$ we define a mapping $A_j$ that associates with a function $f\in W$ a function $A_j(f) \in {\mathcal N}_j$ closest to $f$ in the $\cC$ norm. Then, clearly, for any $f\in W$ we have
\be\label{Lp2}
\|f-A_j(f)\|_\infty \le \de_j.
\ee
We use the mappings $A_j$, $j=1,\dots, J$, to associate with a function $f\in W$ a sequence (a chain) of functions $f_J, f_{J-1},\dots, f_1$ in the following way
$$
f_J := A_J(f),\quad f_j:=A_j(f_{j+1}),\quad j=1,\dots,J-1.
$$
We introduce an auxiliary sequence
\begin{equation}\label{33.3}
\eta_j := \frac{\eta}{4}\frac{2^{(j+1)/2}\e_{2^{j-1}}(W, L_\infty)}{S_J}.
\end{equation}
Then by the definition of $S_J$
\be\label{eta}
\sum_{j=1}^J \eta_j = \eta/4.
\ee

We now proceed to estimating $\mathbb{P}\{\sup_{f\in W}|L_\bX(f)|\ge \eta\}$. First of all, it is easy to see that
$$
|L_\bX(f)-L_\bX(f_J)| \le 2L\de_J \le \eta/2
$$
due to the assumptions \eqref{Lip} and $\de_J\le \eta/(4L)$.
Therefore, if $|L_\bX(f)| \ge \eta$ then $|L_\bX(f_J)|\ge \eta/2$. Using this   and rewriting
$$
L_\bX(f_J) = L_\bX(f_J)-L_\bX(f_{J-1}) +\dots+L_\bX(f_{2})-L_\bX(f_1)+L_\bX(f_1),
$$
we conclude that if $|L_\bX(f)| \ge \eta$ then at least one of the following events occurs:
$$
|L_\bX(f_j)-L_\bX(f_{j-1})|\ge \eta_j\quad\text{for some}\quad j\in \{2, \ldots, J\} \quad\text{or}\quad |L_\bX(f_1)|\ge \eta/4.
$$
Therefore,
\begin{eqnarray}\label{33.5}
\mathbb{P}\{\sup_{f\in W}|L_\bX(f)|\ge\eta\}
\le \mathbb{P}\{\sup_{f\in {\mathcal N}_1}|L_\bX(f)|\ge\eta/4\} \nonumber \\
+\sum_{j=2}^J\sum_{f\in {\mathcal N}_j}
\mathbb{P}\{|L_\bX(f)-L_\bX(A_{j-1}(f))|\ge\eta_j\}\nonumber\\
\le \mathbb{P}\{\sup_{f\in {\mathcal N}_1}|L_\bX(f)|\ge\eta/4\}\nonumber\\
+\sum_{j=2}^J N_j\sup_{f\in W}
\mathbb{P}\{|L_\bX(f)-L_\bX(A_{j-1}(f))|\ge\eta_j\}.
\end{eqnarray}
By our choice of $\delta_j=\e_{2^j}(W, L_\infty)$ we get $N_j\le 2^{2^j} <e^{2^j}$.

We apply  Lemma \ref{Lemma 33.1} with $f_1=f$, $f_2=A_{j-1}(f)$, $\de =\de_{j-1}$, and
$\eta=\eta_j$.
Since $\de_{j-1} > \eta/(4L)$ for $j=1,\dots,J$, then
the condition $\de_{j-1} > \eta_j/(4L)$ follows from \eqref{eta} and the condition
$\|f_1-f_2\|_\infty \le \de_{j-1}$ follows from \eqref{Lp2}.
By Lemma \ref{Lemma 33.1} we obtain
$$
\sup_{f\in W} \mathbb{P}\{|L_\bX(f)-L_\bX(A_{j-1}(f))|\ge \eta_j\} \le 2\exp\Bigl(-\frac{m\eta_j^2}{8L^2\de_{j-1}^2}\Bigr).
$$
From the definition (\ref{33.3}) of $\eta_j$    we get
$$
 \frac{m\eta_j^2}{8L^2\de_{j-1}^2} =\frac{m(\eta/4)^2}{8L^2S_J^2}2^{j+1}= \frac{m(\eta/S_J)^2}{2^7L^2}2^{j+1}.
$$
Using the assumption $m(\eta/S_J)^2 \ge 2^7L^2$, we obtain
$$
\frac{m\eta_j^2}{8L^2\de_{j-1}^2}-2^j = \frac{m(\eta/S_J)^2}{2^7L^2}2^{j+1}-2^j \ge \frac{m(\eta/S_J)^2}{2^7L^2}2^j
$$
and
$$
N_j\exp\left(-\frac{m\eta_j^2}{8L^2\de_{j-1}^2}\right)\le \exp\left(-\frac{m(\eta/S_J)^2}{2^7L^2}2^j\right).
$$
Therefore,
\begin{equation}\label{33.6}
\sum_{j=2}^JN_j\exp\Bigl(-\frac{m\eta_j^2}{8L^2\de_{j-1}^2}\Bigr) \le \exp\Bigl(-\frac{m(\eta/S_J)^2}{2^7L^2}\Bigr).
\end{equation}
By Bernstein's inequality (\ref{31.6}) for $\eta \le L$
\begin{equation}\label{33.7}
\mathbb{P}\{\sup_{f\in {\mathcal N}_1}|L_\bX(f)|\ge \eta/4\} \le 2N_1\exp\Bigl(-\frac{m\eta^2}{60L^2}\Bigr).
\end{equation}
Combining \eqref{33.6} and \eqref{33.7} we obtain
$$
\mathbb{P}\{\sup_{f\in W}|L_\bX(f)|\ge \eta\} \le 9\exp\Bigl(-c(L)m(\eta/{\bar S}_J)^2\Bigr),
$$
where $c(L):= (2^7L^2)^{-1}$ and ${\bar S}_J := \max(1,S_J)$.
This completes the proof of Theorem \ref{LpT1}.
\end{proof}

\begin{Theorem}\label{LpT2} Assume that a function class $W$ satisfies \eqref{M} and is such that
$$
\sum_{n=1}^\infty n^{-1/2}\e_n(W, L_\infty) <\infty.
$$
Denote
$$
S:= \sum_{j=1}^\infty2^{(j+1)/2}\e_{2^{j-1}}(W, L_\infty).
$$
Suppose that $\Phi$ satisfies \eqref{Lip}.
Then for $m\in\bbN$, $\eta\in (0,L]$ satisfying $m(\eta/S)^2 \ge 2^7L^2$ we have
$$
\mathbb{P}\{\sup_{f\in W} |L_\bX(f)|\ge \eta\} \le 9\exp(-c(L)m(\eta/{\bar S})^2)
$$
with $c(L):= (2^7L^2)^{-1}$ and ${\bar S} := \max(1,S)$.
\end{Theorem}
\begin{proof} The proof of this theorem repeats the proof of Theorem \ref{LpT1} with
$S_J$ replaced by $S$. We do not present it here.
\end{proof}

We now derive from Theorem \ref{LpT1} the following Corollary \ref{BC3}.

\begin{Corollary}\label{BC3} Assume that a function class $W$ satisfies \eqref{M} and
$$
\e_n(W, L_\infty)\le n^{-r}(\log (n+1))^b,\qquad r\in (0,1/2),\quad b\ge 0.
$$
Also, assume that $\Phi$ satisfies \eqref{Lip}.
Then there are two positive constants $C:=C(L,r,b)$ and $c:=c(L,r,b)$ such that for $m\in\bbN$, $\eta\in (0,L]$, satisfying
$m \eta^{1/r} (\log (4L/\eta))^{-b/r} \ge C$ we have
\be\label{Co2}
\mathbb{P}\{\sup_{f\in W} |L_\bX(f)|\ge \eta\} \le 9 \exp(- cm\eta^{1/r} (\log (4L/\eta))^{-b/r} ).
\ee
\end{Corollary}
\begin{proof}
We will use the following simple technical lemma from \cite{VT191}.
\begin{Lemma}[{see Lemma 2.1 in \cite{VT191}}]\label{BL1} Let $r>0$, $b\ge 0$, and $A\ge 2$. Then for $n\in \bbN$ the inequality
\be\label{B1}
2^{rn}n^{-b} \le A
\ee
implies inequalities
$$
n\le C_1(r,b)\log A,\qquad 2^n \le C_2(r,b)A^{1/r}(\log A)^{b/r}
$$
with some positive constants $C_i(r,b)$, $i=1,2$.
\end{Lemma}


We now take $\eta \le L$. From the definition of $J$ in Theorem \ref{LpT1} we obtain
$$
\e_{2^J}(W, L_\infty)\le \eta/(4L)\quad \text{and}\quad \e_{2^{J-1}}(W, L_\infty)> \eta/(4L).
$$
Therefore,
$$
2^{-r(J-1)}(J-1)^b > \eta/(4L) \quad \text{and}\quad 2^{r(J-1)}(J-1)^{-b} < 4L/\eta.
$$
By Lemma \ref{BL1} with $A=4L/\eta$ we obtain
\be\label{B2}
J-1\le C_1(r,b)\log (4L/\eta),\qquad 2^{J-1} \le C_2(r,b)(4L/\eta)^{1/r}(\log (4L/\eta))^{b/r}.
\ee
For $S_J$ we obtain the upper bound
$$
S_J= \sum_{j=1}^J2^{(j+1)/2}\e_{2^{j-1}}(W, L_\infty) \le  \sum_{j=1}^J2^{(j+1)/2}2^{-r(j-1)}j^b
$$
\be\label{B3}
\le C_3(r,b) 2^{J(1/2-r)}J^b \le C_4(r,b) (4L/\eta)^{(1/2-r)/r} (\log (4L/\eta))^{b/(2r)}.
\ee
Suppose that $C_4(r,b)$ is large enough to satisfy
$$
S_J':=C_4(r,b) (4L/\eta)^{(1/2-r)/r} (\log (4L/\eta))^{b/(2r)} \ge 1,\quad \eta\in (0,L].
$$
Also, suppose that $C$ is large enough to guarantee $m(\eta/S_J')^2 \ge 2^7L^2$ provided
$m \eta^{1/r} (\log (4L/\eta))^{-b/r} \ge C$.
Then it remains to apply Theorem \ref{LpT1}.
 \end{proof}

\subsection{The second general upper bound for the error of sampling discretization}

We will deduce the main result of this subsection (Theorem \ref{KT1} below) from the following theorem (see \cite[Theorem 1.2]{GMPT07}).

\begin{Theorem}\label{GMPT}
There is a numerical constant $c>0$ such that for any
function class $G\subset \cC(\Omega)$, any
i.i.d. random vectors $X_1, \ldots, X_m$ with the distribution $\mu$, one has
$$
\mathbb{E}\Bigl[\sup\limits_{g\in G}
\Bigl|\frac{1}{m}\sum\limits_{j=1}^m|g(X_j)|^2 - \int_\Omega|g|^2\, d\mu\Bigr|\Bigr]
\le
c\Bigl(A + A^{1/2}\Bigl(\sup\limits_{g\in G}\int_{\Omega}|g|^2\, d\mu\Bigr)^{1/2} \Bigr),
$$
where
$$
A=\frac{1}{m}\mathbb{E}[\gamma_2^2(G, \|\cdot\|_{\infty, {\mathbf{X}}})]
$$
and where $\|g\|_{\infty, {\mathbf{X}}} := \max\limits_{1\le j\le m}|g(X_j)|$.
\end{Theorem}

We point out that the symbol $A$ will usually be used to denote various expressions that are involved in the upper bounds for the expectation
of the error of discretization. The exact values of $A$ may vary from line to line.

\begin{Theorem}\label{KT1}
There is a numerical constant $C>0$ such that, for every $p\in(0 ,+\infty)$,
every compact set $\Omega$ endowed with a probability Borel measure $\mu$, and every $W\subset C(\Omega)$
one has
\begin{multline*}
er_m(W, L_p)\le
\mathbb{E}\Bigl[\sup\limits_{f\in W}
\Bigl|\frac{1}{m}\sum\limits_{j=1}^m|f(X_j)|^p - \|f\|_p^p\Bigr|\Bigr]
\\
\le C\Bigl(A + A^{1/2}\Bigl(\sup\limits_{f\in W}\int_{\Omega}|f|^p\, d\mu\Bigr)^{1/2} \Bigr),
\end{multline*}
where
$$
A:=
\tfrac{\max\{p^2, 1\}}{m}\mathbb{E}\Bigl[\sup\limits_{h\in W}\|h\|_{\infty, {\mathbf{X}}}^{\max\{p-2, 0\}}
\Bigl(\sum\limits_{k\le [\log m]}2^{k/2}\bigl[e_k(W, \|\cdot\|_{\infty, {\mathbf X}})\bigr]^{\frac{\min\{p, 2\}}{2}}\Bigr)^2\Bigr],
$$
where $\|h\|_{\infty, {\mathbf{X}}} := \max\limits_{1\le j\le m}|h(X_j)|$, and
where the random vectors
$X_1, \ldots, X_m$ are i.i.d., and distributed according to the measure $\mu$.
\end{Theorem}

\begin{proof}
By Theorem \ref{GMPT} applied with the set
$$
G:=\{g=|f|^{p/2}\colon f\in W\},
$$
we get
\begin{multline*}
\mathbb{E}\Bigl[\sup\limits_{f\in W}
\Bigl|\frac{1}{m}\sum\limits_{j=1}^m|f(X_j)|^p - \|f\|_p^p\Bigr|\Bigr]
=
\mathbb{E}\Bigl[\sup\limits_{g\in G}
\Bigl|\frac{1}{m}\sum\limits_{j=1}^m|g(X_j)|^2 - \|g\|_2^2\Bigr|\Bigr]
\\
\le c\Bigl(A + A^{1/2}\Bigl(\sup\limits_{g\in G}\int_{\Omega}|g|^2\, d\mu\Bigr)^{1/2} \Bigr),
\end{multline*}
where
$$
A=\frac{1}{m}\mathbb{E}[\gamma_2^2(G, \|\cdot\|_{\infty, {\mathbf{X}}})]
$$
and where $\|g\|_{\infty, {\mathbf{X}}} := \max\limits_{1\le j\le m}|g(X_j)|$.

We now note that
$$
\sup\limits_{g\in G}\int_{\Omega}|g|^2\, d\mu
= \sup\limits_{f\in W}\int_{\Omega}|f|^p\, d\mu.
$$
The set
$$
G_{\mathbf X}:=\bigl\{(g(X_1), \ldots, g(X_m))\colon g\in G\bigr\}
$$
is a subset of an $m$-dimensional space. Thus, we use
Lemma \ref{lem-bound} and obtain
$$
\sum\limits_{k=0}^\infty2^{k/2}e_k(G_{\mathbf X}, \|\cdot\|_{\infty, {\mathbf X}})\le
c\sum\limits_{k\le [\log m]}2^{k/2}e_k(G_{\mathbf X}, \|\cdot\|_{\infty, {\mathbf X}}).
$$

Applying Dudley's entropy bound \eqref{Dudley}, we get
\begin{equation}\label{gamma}
\gamma_2(G, \|\cdot\|_{\infty, {\mathbf{X}}})
\le
c\sum\limits_{k\le [\log m]}2^{k/2}e_k(G_{\mathbf X}, \|\cdot\|_{\infty, {\mathbf X}}).
\end{equation}

Finally, we consider two cases $p\in(0, 2]$ and $p\in(2, \infty)$ separately.

For $p\in (0, 2]$ and $g_1=|f_1|^{p/2}, g_2=|f|^{p/2}\in G$ we have
\begin{multline*}
\|g_1-g_2\|_{\infty, {\mathbf X}}=\max\limits_{1\le j\le m}|g_1(X_j) - g_2(X_j)| = \max\limits_{1\le j\le m}||f_1|^{p/2}(X_j) - |f_2|^{p/2}(X_j)|
\\
\le \bigl(\max\limits_{1\le j\le m}|f_1(X_j) - f_2(X_j)|\bigr)^{p/2} = \|f_1 - f_2\|_{\infty, {\mathbf X}}^{p/2}
\end{multline*}
where we have used the fact that the function $t\mapsto (t+b)^a - t^a$ is non increasing on $[0,\infty)$
for $a=p/2\in (0,1]$, $b\ge 0$.
This estimate implies
$$
e_k(G_{\mathbf X}, \|\cdot\|_{\infty, {\mathbf X}})\le \bigl[e_k(W, \|\cdot\|_{\infty, {\mathbf X}})\bigr]^{p/2}
$$
and, from \eqref{gamma}, we get
$$
\gamma_2(G, \|\cdot\|_{\infty, {\mathbf{X}}})\le
c\sum\limits_{k\le [\log m]}2^{k/2}\bigl[e_k(W, \|\cdot\|_{\infty, {\mathbf X}})\bigr]^{p/2}.
$$

For $p\in (2, \infty)$ and $g_1=|f_1|^{p/2}, g_2=|f|^{p/2}\in G$ we have
\begin{multline*}
\|g_1-g_2\|_{\infty, {\mathbf X}}=\max\limits_{1\le j\le m}|g_1(X_j) - g_2(X_j)|
= \max\limits_{1\le j\le m}||f_1|^{p/2}(X_j) - |f_2|^{p/2}(X_j)|
\\
\le p\max\limits_{1\le j\le m}\bigr(\max\{|f_1(X_j)|^{p/2-1}, |f_2(X_j)|^{p/2-1}\} |f_1(X_j) - f_2(X_j)|\bigr)
\\
\le p\max\limits_{1\le j\le m}\sup\limits_{h\in W}|h(X_j)|^{p/2-1}\|f_1 - f_2\|_{\infty, {\mathbf X}}.
\end{multline*}
This estimate implies
$$
e_k(G_{\mathbf X}, \|\cdot\|_{\infty, {\mathbf X}})\le
p\max\limits_{1\le j\le m}\sup\limits_{h\in W}|h(X_j)|^{p/2-1}\cdot e_k(W, \|\cdot\|_{\infty, {\mathbf X}})
$$
and, from \eqref{gamma}, we get
$$
\gamma_2(G, \|\cdot\|_{\infty, {\mathbf{X}}})\le
cp\max\limits_{1\le j\le m}\sup\limits_{h\in W}|h(X_j)|^{p/2-1}\cdot
\sum\limits_{k\le [\log m]}2^{k/2}e_k(W, \|\cdot\|_{\infty, {\mathbf X}}).
$$
The theorem is proved.
\end{proof}

\section{Universal sampling discretization}
\label{udr}

\subsection{Preliminaries}

We begin this section with a brief description of  some  necessary concepts on sparse approximation.
Let $X$ be a Banach space with a norm $\|\cdot\|:=\|\cdot\|_X$, and let $\D=\{g_i\}_{i=1}^\infty $ be a fixed (countable)
system of elements in $X$. We recall (see \eqref{sigma-v}) that
$$
\Sigma_v(\D):= \bigcup_{j_1<\ldots<j_v}\sp\{g_{j_1}, \ldots, g_{j_v}\}.
$$
For a given $f\in X$, we define
$$
\sigma_v(f,\D)_X := \inf_{g\in\Sigma_v(\D)}\|f-g\|_X,\  \ v=1,2,\cdots, \quad \sigma_0(f,\D)_X:= \|f\|_X.
$$
Moreover, for a function class $W\subset X$, we define
$$
\sigma_v(W,\D)_X := \sup_{f\in W} \sigma_v(f,\D)_X,\quad  \sigma_0(W,\D)_X := \sup_{f\in W} \|f\|_X.
$$
In the case $X=L_q$ we use a short form $\sigma_v(W,\D)_q$ instead of $\sigma_v(W,\D)_{L_q}$.
For a uniformly bounded system $\D = \{g\}$  define
$$
A_1(\D) := \Bigl\{f:\, f=\sum_{i=1}^\infty a_ig_i,\quad g_i\in \D,\quad \sum_{i=1}^\infty |a_i| \le 1 \Bigr\}.
$$
Recall that  the modulus of smoothness of a Banach space  $X$ is defined as
\begin{equation}\label{CG1}
\eta(w):=\eta(X,w):=\sup_{\sub{x,y\in X\\
		\|x\|= \|y\|=1}}\left(\f {\|x+wy\|+\|x-wy\|}2 -1\right),\  \ w>0,
\end{equation}
and that $X$ is called uniformly smooth  if  $\eta(w)/w\to 0$ when $w\to 0+$.
It is well known that the $L_p$ space with $1< p<\infty$ is a uniformly smooth Banach space with
\be\label{CG2}
\eta(L_p,w)\le \begin{cases}(p-1)w^2/2, & 2\le p <\infty,\\   w^p/p,& 1\le p\le 2.
\end{cases}
\ee

The following bound follows from the known results about greedy approximation algorithms
(see 
Theorem 8.6.6, Remark 8.6.10, and Theorem 9.2.1, in \cite{VTbookMA}).

\begin{Theorem}\label{saT1}
Let $X$ be a Banach space satisfying that  $\eta(X, w)\le \gamma w^r$, $w>0$ for some parameter $1<r\le 2$.
Then, for any system $\D = \{g\}$, $\|g\|_X \le 1$,  we have
\be\label{sa3}
\sigma_v(A_1(\D),\D)_X \le C\gamma^{1/r}(v+1)^{1/r-1}.
\ee
\end{Theorem}

Note that Theorem \ref{saT1} was proved for a dictionary $\D$ but that proof works for a system as well.

We recall the following definition of the important Nikol'skii inequality condition.
\begin{Definition}\label{ID2}
Let $1\le q<\infty$. We say that a system $\D_N:=\{g_i\}_{i=1}^N$ satisfies the Nikol'skii inequality for the pair $(q,\infty)$ with the constant $H$ if for any $f\in \sp\{\D_N\}$ we have
$$
\|f\|_\infty \le H\|f\|_q.
$$
We write $\D_N \in NI_q(H)$ for systems $\D_N$ satisfying this condition.
\end{Definition}

\begin{Corollary}\label{saT2} Let $q\in[2,\infty)$.
Assume that $\D_N=\{g_1, \ldots, g_N\}$ is a family of bounded functions on some set ${\mathbf X}$
endowed with some probability measure $\nu$ such that
$\sup\limits_{x\in {\mathbf X}}|g_j(x)| \le 1$, $j=1,\dots,N$,
and $\D_N \in NI_q(H)$ (with respect to the measure $\nu$). Then we have
\be\label{sigm}
\sigma_v(A_1(\D_N),\D_N)_\infty \le CHq^{1/2} (v+1)^{-1/2}
\ee
for some numerical constant $C>0$.
\end{Corollary}
\begin{proof} Our assumption $\|g_i\|_\infty \le 1$, $i=1,\dots,N$, implies that $\|g_i\|_{L_q(\nu)} \le 1$, $i=1,\dots,N$.
Therefore, by Theorem \ref{saT1} and inequalities (\ref{CG2}) we obtain
\be\label{sa4}
\sigma_v(A_1(\D_N),\D_N)_q \le   Cq^{1/2}(v+1)^{-1/2}.
\ee
By our assumption $\D_N \in NI_q(H)$ we find
\be\label{sa5}
\sigma_v(A_1(\D_N),\D_N)_\infty\le H\sigma_v(A_1(\D_N),\D_N)_q \le   CHq^{1/2}(v+1)^{-1/2}.
\ee
The corollary is proved.
\end{proof}

To be able to apply Theorem \ref{KT1},
we need to have bounds for the entropy numbers with respect to the discrete
uniform norm $\|f\|_{\infty, {\mathbf X}}:=\max\limits_{x\in {\mathbf X}}|f(x)|$, ${\mathbf X}=\{X_1, \ldots, X_m\}$.
To obtain such bounds we will use the following
theorem (see \cite{VT138} and Theorem 7.4.3 in \cite{VTbookMA}).

\begin{Theorem}\label{enT1} Let $X$ be a Banach space.
Let a compact set $W\subset X$ be such that there exists a
system $\D_N$, $|\D_N|=N$, and  a number $r>0$ such that
$$
  \sigma_{m-1}(W,\D_N)_X \le m^{-r},\quad m\le N.
$$
Then for $k\le N$
\begin{equation}\label{A3}
\e_k(W,X) \le C(r) \left(\frac{\log(2N/k)}{k}\right)^r.
\end{equation}
\end{Theorem}

\begin{Corollary}\label{enT2} Let $q\in[2,\infty)$.
Assume that $\D_N=\{g_1, \ldots, g_N\}$ is a family of bounded functions on some set ${\mathbf X}$
endowed with some probability measure $\nu$ such that
$\sup\limits_{x\in {\mathbf X}}|g_j(x)| \le 1$, $j=1,\dots,N$,
and $\D_N \in NI_q(H)$ (with respect to the measure $\nu$). Then we have
\be\label{ep}
\e_k(A_1(\D_N), L_\infty({\mathbf X})) \le CHq^{1/2}[\log N]^{1/2} k^{-1/2}\quad k=1,2,\dots
\ee
for some numerical constant $C>0$.
\end{Corollary}

\begin{proof} By Corollary \ref{saT2} and Theorem \ref{enT1} with $r=1/2$ we obtain
$$
 \e_k(A_1(\D_N),L_\infty({\mathbf X})) \le CHq^{1/2} k^{-1/2} [\log N]^{1/2}\quad k=1,2,\dots, N.
$$
We derive from here the inequality (\ref{ep}) for all $k>N$.  For $k>N$ we use the inequalities
(see estimate (7.1.6) in \cite{VTbookMA} and Corollary 7.2.2 there)
$$
\e_k(W,L_\infty({\mathbf X})) \le \e_N(W,L_\infty({\mathbf X}))\e_{k-N}(X_N^\infty,L_\infty({\mathbf X}))
$$
and, for $X_N^\infty:=\{f\in  \sp\{\D_N\}\colon \sup\limits_{x\in {\mathbf X}}|f(x)| \le 1\}$,
$$
\e_{n}(X_N^\infty,L_\infty({\mathbf X})) \le 3(2^{-n/N}),\qquad 2^{-x} \le 1/x,\quad x\ge 1,
$$
to obtain (\ref{ep}) for all $k$.
This completes the proof.
\end{proof}

We will also need the following corollary of Lemma 3.1 from \cite{DT}.

\begin{Proposition}\label{Prop-entropy}
Let $p\in(0, 2)$.
Let $\D_N=\{g_1, \ldots, g_N\}$ be a family of continuous functions on a compact set $\Omega$
endowed with a probability measure $\mu$.
Assume that
$$
\e_k(\Sigma_{2v}^2(\D_N), L_\infty(\Omega))\le Bk^{-1/2}\quad k=1,2,\dots
$$
Then
$$
\e_k(\Sigma_v^p(\D_N), L_\infty(\Omega))\le C(p)B^{2/p}k^{-1/p}\quad k=1,2,\dots
$$
where $\Sigma_v^p(\D_N):=\{f\in \Sigma_v(\D_N) \colon \|f\|_{L_p(\mu)}\le 1\}$ and where
$C(p)$ is a positive constant, depending only on $p$.
\end{Proposition}

We note that Lemma 3.1 was proved in \cite{DT} only for $p\in [1, 2)$,
but a straightforward modification of the argument gives a similar result for all $p\in (0, 2)$.

\subsection{Universal sampling discretization for $p\in (0, 2]$}

We are now ready to provide universal discretization results.
We start with the case $p=2$.

\begin{Theorem}\label{deT-1}
Let $K\ge 1$, $\varepsilon\in(0, 1/2]$. There is a numerical constant $C>0$ such that,
for any probability measure $\mu$ on a compact set $\Omega$,
for any system $\D_N=\{g_1, \ldots, g_N\}\subset C(\Omega)$,
which satisfies conditions: (1) $\|g_i\|_\infty \le 1$, $i=1,\dots,N$, (2)
condition \eqref{Cond2}, for any
$$
m\ge C\varepsilon^{-2}[\log\varepsilon^{-1}]^{3}Kv\log N[\log 2Kv + \log\log 2N]^3,
$$
there are points $\{\xi^1, \ldots, \xi^m\}\subset \Omega$ such that
$$
(1-\varepsilon)\|f\|_2^2\le \frac{1}{m}\sum_{j=1}^{m}|f(\xi^j)|^2\le (1+\varepsilon)\|f\|_2^2
\quad \forall f\in \Sigma_v(\D_N).
$$
\end{Theorem}

\begin{proof}
Recall that $\Sigma_v^2(\D_N):=\{f\in \Sigma_v(\D_N)\colon \|f\|_{L_2(\mu)}\le 1\}$.
Let
$$
g=\alpha_1g_{j_1}+\ldots+\alpha_vg_{j_v}\in \Sigma_v^2(\D_N).
$$
We note that
$$
|\alpha_1|+\ldots+|\alpha_v|\le v^{1/2}\bigl(|\alpha_1|^2+\ldots+|\alpha_v|^2\bigr)^{1/2}
\le K^{1/2}v^{1/2}\|g\|_2.
$$
Thus, $g\in K^{1/2}v^{1/2} A_1(\D_N)$, i.e. $\Sigma_v^2(\D_N)\subset K^{1/2}v^{1/2} A_1(\D_N)$.
Now, for any fixed set $\mathbf{X}:=\{X_1, \ldots, X_m\}$ of $m$ points,
we have
$$
\max_{1\le j\le m}|f(X_j)|\le m^{1/q}\Bigl(\frac{1}{m}\sum_{j=1}^{m}|f(X_j)|^q\Bigr)^{1/q}\quad
\forall f\in \mathrm{ span}\{\D_N\}.
$$
Therefore,
we can apply Corollary \ref{enT2}
with the measure $\nu=\frac{1}{m}\sum_{j=1}^{m}\delta_{X_j}$, $q = \log m$, $H=2$, and obtain
\begin{align}\label{eq-ent-bound}
\e_k(\Sigma_v^2(\D_N), L_\infty(\mathbf{X}))\le& K^{1/2}v^{1/2}\e_k(A_1(\D_N), L_\infty(\mathbf{X}))
\nonumber
\\
\le& C_1 K^{1/2}v^{1/2} [\log m]^{1/2} [\log N]^{1/2} k^{-1/2} \quad \forall k\in \mathbb{N}.
\end{align}
By Theorem \ref{KT1} we have
\begin{align}\label{eq2}
&er_m\bigl(\Sigma_v^2(\D_N), L_2\bigr) \nonumber
\\
&\le
\mathbb{E}\Bigl[\sup\limits_{f\in \Sigma_v^2(\D_N)}
\Bigl|\frac{1}{m}\sum\limits_{j=1}^m|f(X_j)|^2 - \|f\|_2^2\Bigr|\Bigr]
\le C_2\bigl(A + A^{1/2}\bigr),
\end{align}
where
$$
A=\frac{Kv[\log m]^3[\log N]}{m}.
$$
It is now sufficient to take $m$ such that $\max\{1, C_2\}\sqrt{A}\le \varepsilon/2$,
or, equivalently,
$$
\varepsilon^{-2}\frac{Kv[\log m]^3[\log N]}{m}\le \frac{1}{4\max\{1, C_2^2\}}.
$$
It is clear that any
$$
m\ge C\varepsilon^{-2}[\log \varepsilon^{-1}]^3Kv\log N[\log 2Kv + \log\log 2N]^3
$$
with sufficiently large $C>0$ will be suitable.
\end{proof}

For the case $p\in (0, 2)$ we firstly prove the following lemma.
Its proof almost repeats the proof of the previous theorem.

\begin{Lemma}\label{lem-deT}
Let $p\in(0, 2)$, $K\ge 1$, $\varepsilon\in(0, 1/2]$. There is a constant $C(p)>0$, depending only on $p$, such that,
for any probability measure $\mu$ on a compact set $\Omega$,
for any system $\D_N=\{g_1, \ldots, g_N\}\subset C(\Omega)$,
which satisfies conditions: (1) $\|g_i\|_\infty \le 1$, $i=1,\dots,N$, (2)
condition \eqref{Cond2} one has
$$
\mathbb{E}\Bigl[\sup\limits_{f\in \Sigma_v^p(\D_N)}
\Bigl|\frac{1}{m}\sum\limits_{j=1}^m|f(X_j)|^p - \|f\|_p^p\Bigr|\Bigr]
\le C(p)(B^{p/2}+B^{p/4}),
$$
where
$$
B=\frac{(Kv)^{2/p}[\log m][\log N]}{m},
$$
$\Sigma_v^p(\D_N):=\{f\in \Sigma_v(\D_N)\colon \|f\|_{L_p(\mu)}\le 1\}$,
and random vectors
$X_1, \ldots, X_m$ are i.i.d., and distributed according to the measure $\mu$.
\end{Lemma}

\begin{proof}
Let
$$
g=\alpha_1g_{j_1}+\ldots+\alpha_vg_{j_v}\in \Sigma_v^p(\D_N).
$$
We note that
\begin{multline*}
|\alpha_1|+\ldots+|\alpha_v|\le v^{1/2}\bigl(|\alpha_1|^2+\ldots+|\alpha_v|^2\bigr)^{1/2}
\\
\le K^{1/2}v^{1/2}\|g\|_2\le K^{1/2}v^{1/2}\|g\|_p^{p/2}(|\alpha_1|+\ldots+|\alpha_v|)^{1-p/2}.
\end{multline*}
Thus, $g\in K^{1/p}v^{1/p} A_1(\D_N)$, i.e. $\Sigma_v^p(\D_N)\subset K^{1/p}v^{1/p} A_1(\D_N)$.
For any fixed set $\mathbf{X}:=\{X_1, \ldots, X_m\}$ of $m$ points, we again apply Corollary \ref{enT2}
with $\nu=\frac{1}{m}\sum_{j=1}^{m}\delta_{X_j}$, $q = \log m$, $H=2$, and get
\begin{multline*}
\e_k(\Sigma_v^p(\D_N), L_\infty(\mathbf{X}))\le K^{1/p}v^{1/p}\e_k(A_1(\D_N), L_\infty(\mathbf{X}))
\\
\le C_1 K^{1/p}v^{1/p} [\log m]^{1/2} [\log N]^{1/2} k^{-1/2} \quad \forall k\in \mathbb{N}.
\end{multline*}
By Theorem \ref{KT1} we have
$$
\mathbb{E}\Bigl[\sup\limits_{f\in \Sigma_v^p(\D_N)}
\Bigl|\frac{1}{m}\sum\limits_{j=1}^m|f(X_j)|^p - \|f\|_p^p\Bigr|\Bigr]
\le C\bigl(A + A^{1/2} \bigr),
$$
where
\begin{multline*}
A=m^{-1}\mathbb{E}\Bigl[
\Bigl(\sum\limits_{k\le [\log m]}2^{k/2}\bigl[e_k(\Sigma_v^p(\D_N), \|\cdot\|_{\infty, {\mathbf X}})\bigr]^{p/2}\Bigr)^2\Bigr]
\\
\le
m^{-1}Kv[\log m]^{p/2} [\log N]^{p/2}
\Bigl(\sum\limits_{k\le [\log m]}2^{k/2}2^{-pk/4}\Bigr)^2
\\
\le C_2(p)m^{-p/2}Kv[\log m]^{p/2} [\log N]^{p/2} = C_2(p)\Bigl(\frac{K^{2/p}v^{2/p}[\log m][\log N]}{m}\Bigr)^{p/2}.
\end{multline*}
The lemma is proved.
\end{proof}

\begin{Theorem}\label{deT-2}
Let $p\in(0, 2)$, $K\ge 1$, $\varepsilon\in(0, 1/2]$. There is a constant $C:=C(p)>0$, depending only on $p$, such that,
for any probability measure $\mu$ on a compact set $\Omega$,
for any system $\D_N=\{g_1, \ldots, g_N\}\subset C(\Omega)$,
which satisfies conditions: (1) $\|g_i\|_\infty \le 1$, $i=1,\dots,N$, (2)
condition \eqref{Cond2}, there are
$$
m\le C\varepsilon^{-2}[\log\varepsilon^{-1}]^{3}Kv\log N[\log 2Kv + \log\log 2N]^3,
$$
points $\{\xi^1, \ldots, \xi^m\}\subset \Omega$ such that
$$
(1-\varepsilon)\|f\|_p^p\le \frac{1}{m}\sum_{j=1}^{m}|f(\xi^j)|^p\le (1+\varepsilon)\|f\|_p^p
\quad \forall f\in \Sigma_v(\D_N).
$$
\end{Theorem}

\begin{proof}
By Lemma \ref{lem-deT} and Theorem \ref{deT-1} (see \eqref{eq2} in the proof there)
we have
$$
\mathbb{E}\Bigl[\sup\limits_{f\in \Sigma_v^p(\D_N)}
\Bigl|\frac{1}{m}\sum\limits_{j=1}^m|f(X_j)|^p - \|f\|_p^p\Bigr|\Bigr]
\le C_1(p)(B^{p/2}+B^{p/4}),
$$
where
$$
B=\frac{(Kv)^{2/p}[\log m][\log N]}{m}\le c_1\frac{(Kv)^{2/p}[\log N]}{\sqrt{m}},
$$
and
$$
\mathbb{E}\Bigl[\sup\limits_{f\in \Sigma_{2v}^2(\D_N)}
\Bigl|\frac{1}{m}\sum\limits_{j=1}^m|f(X_j)|^2 - \|f\|_2^2\Bigr|\Bigr]
\le C_2\bigl(A + A^{1/2}\bigr),
$$
where
$$
A=\frac{2Kv[\log m]^3[\log N]}{m}\le c_2\frac{Kv[\log N]}{\sqrt{m}}.
$$
Let $m_0 = C_3(p)\varepsilon^{-8/p}(Kv)^{4/p} [\log N]^2$
with sufficiently large constant $C_3(p)$ such that
$$
c_1\varepsilon^{-4/p}\frac{(Kv)^{2/p}[\log N]}{\sqrt{m_0}} = \frac{c_1}{\sqrt{C_3(p)}}
\le
\frac{1}{[32\max\{1, C_1(p)\}]^{4/p}}
$$
and, similarly,
$$
c_2\frac{Kv[\log N]}{\sqrt{m_0}}
=
c_2\varepsilon^{4/p}(Kv)^{1-2/p}\frac{1}{\sqrt{C_3(p)}}
\le \frac{c_2}{\sqrt{C_3(p)}}
\le \frac{1}{[16\max\{1, C_2\}]^2}.
$$
In that case we have
$$
\mathbb{E}\Bigl[\sup\limits_{f\in \Sigma_v^p(\D_N)}
\Bigl|\frac{1}{m_0}\sum\limits_{j=1}^{m_0}|f(X_j)|^p - \|f\|_p^p\Bigr|\Bigr]
\le \varepsilon/16,
$$
and
$$
\mathbb{E}\Bigl[\sup\limits_{f\in \Sigma_{2v}^2(\D_N)}
\Bigl|\frac{1}{m_0}\sum\limits_{j=1}^{m_0}|f(X_j)|^2 - \|f\|_2^2\Bigr|\Bigr]
\le 1/8.
$$
Thus,
$$
\mathbb{P}\Bigl(\sup\limits_{f\in \Sigma_v^p(\D_N)}
\Bigl|\frac{1}{m_0}\sum\limits_{j=1}^{m_0}|f(X_j)|^p - \|f\|_p^p\Bigr|\le\varepsilon/4\Bigr)\ge 1- 1/4 = 3/4
$$
and
$$
\mathbb{P}\Bigl(\sup\limits_{f\in \Sigma_{2v}^2(\D_N)}
\Bigl|\frac{1}{m_0}\sum\limits_{j=1}^{m_0}|f(X_j)|^2 - \|f\|_2^2\Bigr|\le1/2\Bigr)\ge 1- 1/4 = 3/4.
$$
Therefore, there is a choice of points $\mathbf{x}:=\{x_1, \ldots, x_{m_0}\}\subset \Omega$
such that
\begin{equation}\label{eq-p-1}
(1-\varepsilon/4)\|f\|_p^p\le \frac{1}{m_0}\sum\limits_{j=1}^{m_0}|f(x_j)|^p\le (1+\varepsilon/4)\|f\|_p^p\quad \forall f\in \Sigma_v(\D_N)
\end{equation}
and
\begin{equation}\label{eq-p-2}
\frac{1}{2}\|f\|_2^2\le \frac{1}{m_0}\sum\limits_{j=1}^{m_0}|f(x_j)|^2\le \frac{3}{2}\|f\|_2^2\quad \forall f\in \Sigma_{2v}(\D_N).
\end{equation}
We now consider new (discrete) measure $\nu:=\frac{1}{m_0}\sum_{j=1}^{m_0}\delta_{x_j}$ on $\Omega$.
Repeating the argument from the proof of Theorem \ref{deT-1} we get
$$
\Sigma_{2v}^2(\D_N)\subset (2Kv)^{1/2}A_1(\D_N).
$$
From \eqref{eq-p-2} we get
$$
\Sigma_{2v}^2(\D_N, \nu):=\Sigma_{2v}(\D_N)\cap\{f\in \sp(\D_N)\colon \|f\|_{L_2(\nu)}\le 1\}\subset \sqrt{2}\Sigma_{2v}^2(\D_N)
$$
and proceeding as in the proof of Theorem \ref{deT-1} by applying Corollary \ref{enT2}
with the measure $\nu=\frac{1}{m_0}\sum_{j=1}^{m_0}\delta_{x_j}$, $q = \log m_0$, $H=2$, we conclude that
\begin{multline*}
\e_k(\Sigma_{2v}^2(\D_N, \nu), L_\infty(\mathbf{x}))\le 2(Kv)^{1/2}\e_k(A_1(\D_N), L_\infty(\mathbf{x}))
\\
\le C_4 (Kv)^{1/2} [\log m_0]^{1/2} [\log N]^{1/2} k^{-1/2}
\\
\le C_5(p) [\log\varepsilon^{-1}]^{1/2}(Kv)^{1/2}[\log N]^{1/2}[\log 2Kv + \log\log N]^{1/2} k^{-1/2}
\quad \forall k\in \mathbb{N}.
\end{multline*}
By Proposition \ref{Prop-entropy} we have
\begin{multline*}
\e_k(\Sigma_v^p(\D_N), L_\infty(\mathbf{x}))
\\
\le C_6(p)[\log\varepsilon^{-1}]^{1/p}(Kv)^{1/p}[\log N]^{1/p}[\log 2Kv + \log\log N]^{1/p}
k^{-1/p}\quad k\in \mathbb{N}.
\end{multline*}
By Theorem \ref{KT1} we have
$$
\mathbb{E}\Bigl[\sup\limits_{f\in \Sigma_v^p(\D_N, \nu)}
\Bigl|\frac{1}{m}\sum\limits_{j=1}^m|f(X_j)|^p - \|f\|_{L_p(\nu)}^p\Bigr|\Bigr]
\le C\bigl(A + A^{1/2} \bigr),
$$
where
\begin{multline*}
A=m^{-1}\mathbb{E}\Bigl[
\Bigl(\sum\limits_{k\le [\log m]}2^{k/2}\bigl[e_k(\Sigma_v^p(\D_N, \nu), \|\cdot\|_{\infty, {\mathbf X}})\bigr]^{p/2}\Bigr)^2\Bigr]
\\
\le
C_7(p)m^{-1}[\log m]^2[\log\varepsilon^{-1}]Kv[\log N][\log 2Kv + \log\log N]
\end{multline*}
We now choose $m = C_8(p)\varepsilon^{-2} [\log\varepsilon^{-1}]^3Kv[\log N][\log 2Kv + \log\log N]^3$
with sufficiently large constant $C_8(p)$ such that
$$
\mathbb{E}\Bigl[\sup\limits_{f\in \Sigma_v^p(\D_N, \nu)}
\Bigl|\frac{1}{m}\sum\limits_{j=1}^m|f(X_j)|^p - \|f\|_{L_p(\nu)}^p\Bigr|\Bigr]
\le \varepsilon/8
$$
implying the existence of the set $\{\xi^1, \ldots, \xi^m\}\subset \mathbf{x}:=\{x_1, \ldots, x_{m_0}\}$
such that
$$
(1-\varepsilon/4)\|f\|_{L_p(\nu)}^p\le \frac{1}{m}\sum_{i=1}^{m}|f(\xi^i)|^p\le (1+\varepsilon/4)\|f\|_{L_p(\nu)}^p\quad
\forall f\in \Sigma_v^p(\D_N).
$$
Thus,
$$
(1-\varepsilon/4)^2\|f\|_p^p\le \frac{1}{m}\sum_{i=1}^{m}|f(\xi^i)|^p\le (1+\varepsilon/4)^2\|f\|_p^p\quad
\forall f\in \Sigma_v^p(\D_N)
$$
and we notice that $1-\varepsilon\le (1-\varepsilon/4)^2$ and $(1+\varepsilon/4)^2\le 1+\varepsilon$.
The theorem is proved.
\end{proof}

\begin{Remark}
The assumption that the system $\D_N=\{g_i\}_{i=1}^N$
satisfies condition \eqref{Cond2} can be replaced by the assumption
that there is a constant $K\ge 1$ such that,
for any $(a_{j_1}, \ldots, a_{j_{2v}})\in \mathbb{C}^{2v}$,
one has
\begin{equation}\label{Cond2'}
\sum_{k=1}^{2v}|a_{j_k}|\le K^{1/2}v^{1/2}\Bigl\|\sum_{k=1}^{2v}a_{j_k}g_{j_k}\Bigr\|_2
\end{equation}
for any choice of $\{j_1, \ldots, j_{2v}\}\subset \{1, \ldots, N\}$.
\end{Remark}

Note that conditions of the type (\ref{Cond2'}) are useful in the theory of Lebesgue-type inequalities for greedy algorithms (see \cite{VTbookMA}, Section 8.7).

\vskip .2in

{\bf Acknowledgements.}
The research of the first named author (Section~3) was supported
by the Marie Sklodowska-Curie grant 101109701 and by the Spanish State Research Agency, through the Severo Ochoa and Mar\'ia de Maeztu Program for Centers and Units of Excellence in R\&D (CEX2020-001084-M).
The first named author thanks CERCA Programme (Generalitat de Catalunya) for institutional support.

The research of the second named author (Section~2)
was supported by the Russian Science Foundation (project No. 23-71-30001)
at the Lomonosov Moscow State University.

\vskip .1in

{\bf Competing interests.}
The authors have no competing interests to declare that are relevant to the content of this article.

\Addresses


\begin{thebibliography}{9999}

\bibitem{BSU} F.~Bartel, M.~Sch\"afer, T.~Ullrich,
Constructive subsampling of finite frames with applications in optimal function recovery, 
{\em Appl. Comput. Harmon. Anal.} {\bf 65} (2023) 209--248.
https://doi.org/10.1016/j.acha.2023.02.004

\bibitem{Ca} 
B. Carl, Entropy numbers, $s$-numbers, and eigenvalue problems, {\em J. Funct. Anal.} {\bf 41} (1981) 290--306.
https://doi.org/10.1016/0022-1236(81)90076-8

\bibitem{CoDo} A. Cohen and M. Dolbeault, Optimal pointwise sampling for $L^2$ approximation,
{\it  J. Complexity.} {\bf  68} (2022) 101602.
https://doi.org/10.1016/j.jco.2021.101602

\bibitem{CM} A. Cohen and G. Migliorati, Optimal weighted least-squares methods, 
{\it SMAI J. Comput. Math.} {\bf 3} (2017) 181--203.
https://doi.org/10.5802/smai-jcm.24

\bibitem{DPTT}
F. Dai, A. Prymak, V.N. Temlyakov, and  S.U. Tikhonov,
Integral norm discretization and related problems, {\em Russian Math. Surveys.} {\bf 74}:4 (2019) 579--630.
https://doi.org/10.1070/RM9892

\bibitem{DKT} F. Dai, E. Kosov, and V. Temlyakov,
Some improved bounds in sampling discretization of integral norms,
{\em J. Funct. Anal.} {\bf285} (2023) 109951.
https://doi.org/10.1016/j.jfa.2023.109951

\bibitem{DT22}
F. Dai and V. Temlyakov, Sampling discretization of integral norms and its application,
{\em Proc. Steklov Inst. Math.} {\bf 319} (2022) 97--109.
https://doi.org/10.1134/S0081543822050091

\bibitem{DT} F. Dai and V.N.  Temlyakov,
Universal sampling discretization, {\em Constr. Approx.} 58 (2023) 589--613.
https://doi.org/10.1007/s00365-023-09644-2.

\bibitem{DTM1} F.  Dai and V.N.  Temlyakov, Universal discretization and sparse sampling recovery,
arXiv:2301.05962v1.

\bibitem{DTM2} F. Dai and V.N. Temlyakov, Random points are good for universal discretization,
{\em J. Math. Anal. Appl.} {\bf 529} (2024) 127570.
https://doi.org/10.1016/j.jmaa.2023.127570

\bibitem{DTM3} F. Dai and V.N. Temlyakov, Lebesgue-type inequalities in sparse sampling recovery,
arXiv:2307.04161v1.


\bibitem{DTU} Ding D{\~u}ng, V.N. Temlyakov, and T. Ullrich,
Hyperbolic Cross Approximation,
Advanced Courses in Mathematics CRM Barcelona,
Birkh{\"a}user, 2018.


\bibitem{DKU} M. Dolbeault, D. Krieg, and M. Ullrich, A sharp upper bound for sampling numbers in $L_2$,  
{\it Appl. Comput. Harmon. Anal.} {\bf 63} (2023) 113--134.
https://doi.org/10.1016/j.acha.2022.12.001



\bibitem{GMPT07}
O. Gu\'edon, S. Mendelson, A. Pajor, and N. Tomczak-Jaegermann,
Subspaces and orthogonal decompositions generated by bounded orthogonal systems.
{\em Positivity} {\bf 11} (2007) 269--283.
https://doi.org/10.1007/s11117-006-2059-1

\bibitem{JUV} T. Jahn, T. Ullrich, and F. Voigtlaender, Sampling numbers of smoothness classes via
$\ell^1$-minimization, {\it J. Complexity} {\bf  79} (2023) 101786.
https://doi.org/10.1016/j.jco.2023.101786

\bibitem{KUV} 
L. K{\"a}mmerer, T. Ullrich, and T. Volkmer, Worst-case recovery guarantees for least squares approximation using random samples,
{\it Constr. Approx.} {\bf 54} (2021) 295--352.	
https://doi.org/10.1007/s00365-021-09555-0

\bibitem{KKLT} B.S. Kashin, E.D. Kosov, I.V. Limonova, and V.N. Temlyakov, Sampling discretization and related problems,
{\em J. Complexity.} {\bf 71} (2022) 101653.
https://doi.org/10.1016/j.jco.2022.101653


\bibitem{Kos}
E.D. Kosov, Marcinkiewicz-type discretization of $L^p$-norms under the Nikolskii-type inequality assumption,
{\em J. Math. Anal. Appl.} {\bf504} (2021) 125358.
https://doi.org/10.1016/j.jmaa.2021.125358

\bibitem{KT24}
E. Kosov and S. Tikhonov, Sampling discretization theorems in Orlicz spaces, arXiv:2406.03444.

\bibitem{KU} D.~Krieg and M.~Ullrich, Function values are enough for $L_2$-approximation, 
{\it  Found. Comput. Math.} {\bf  21} (2021) 1141--1151.
https://doi.org/10.1007/s10208-020-09481-w

\bibitem{KU2} D.~Krieg and M.~Ullrich, Function values are enough for $L_2$-approximation: Part II,  
{\it J. Complexity.} {\bf  66} (2021) 101569.
https://doi.org/10.1016/j.jco.2021.101569

\bibitem{KPUU} 
D.~Krieg, K.~Pozharska, M.~Ullrich, and T.~Ullrich, Sampling recovery in $L_2$ and other norms, (2023), arXiv:2305.07539v3.

\bibitem{KPUU2} 
D.~Krieg, K.~Pozharska, M.~Ullrich, and T.~Ullrich, Sampling projections in the uniform norm, (2024), arXiv:2401.02220.

\bibitem{KNU24}
D.~Krieg, E.~Novak, and M.~Ullrich, On the power of adaption and randomization, (2024), arXiv:2406.07108.

\bibitem{LedTal} M. Ledoux and M. Talagrand, Probability in Banach Spaces: isoperimetry and processes,
Springer Science and Business Media, 2013.

\bibitem{LT}
I. Limonova and V. Temlyakov,
On sampling discretization in $L_2$,
{\em J. Math. Anal. Appl.} {\bf 515} (2022) 126457.
https://doi.org/10.1016/j.jmaa.2022.126457

\bibitem{LMT} I. Limonova, Yu. Malykhin, and V. Temlyakov, One-sided discretization inequalities and sampling recovery, 
{\em Uspehi Matem. Nauk.} {\bf 79} (2024) 149--180. 
https://doi.org/10.4213/rm10175

\bibitem{NSU} N. Nagel, M. Sch{\"a}fer, T. Ullrich, A new upper bound for sampling numbers, 
{\it Found. Comput. Math.} {\bf  22} (2022) 445--468.
https://doi.org/10.1007/s10208-021-09504-0

\bibitem{Tal}
M. Talagrand, Upper and lower bounds for stochastic processes: modern methods and classical problems,
Springer, Berlin, Heidelberg, 2014.



\bibitem{VTbook} V.N. Temlyakov, Greedy Approximation, Cambridge University Press, 2011.

\bibitem{VT138} V.N. Temlyakov, An inequality for the entropy numbers and its application,
{\em J. Approx. Theory.} {\bf 173} (2013) 110--121.
https://doi.org/10.1016/j.jat.2013.05.003




\bibitem{VTbookMA} V.N. Temlyakov, Multivariate Approximation, Cambridge University Press, 2018.

\bibitem{VT171} V.N. Temlyakov, Sampling discretization error for integral norms for function classes,
{\em J. Complexity.} {\bf 54} (2019) 101408.
https://doi.org/10.1016/j.jco.2019.05.002

\bibitem{VT183} V. Temlyakov, On optimal recovery in $L_2$, {\it J. Complexity.} {\bf 65} (2021) 101545.
https://doi.org/10.1016/j.jco.2020.101545

\bibitem{VT184} V.N. Temlyakov and T. Ullrich, Bounds on Kolmogorov widths of classes with small mixed smoothness,
{\em J. Complexity.} {\bf 67} (2021) 101575.
https://doi.org/10.1016/j.jco.2021.101575

\bibitem{VT185} V.N. Temlyakov and T. Ullrich, Approximation of functions with small mixed smoothness in the uniform norm,
{\em J. Approx. Theory.} {\bf 277} (2022) 105718.
https://doi.org/10.1016/j.jat.2022.105718

\bibitem{VT191} V.N. Temlyakov,  Sampling discretization error of integral norms for function classes with small smoothness,
{\em J. Approx. Theory.} {\bf 293} (2023) 105913.
https://doi.org/10.1016/j.jat.2023.105913





\end{thebibliography}
\end{document}